\documentclass[10pt]{amsart}

\usepackage{amsmath,amsthm}
\usepackage[colorlinks=true,linkcolor=blue,urlcolor=blue,citecolor=blue]{hyperref}
\usepackage{amsrefs}
\usepackage[all]{xy}

\usepackage{amsrefs} 
\usepackage{enumerate} 
\usepackage{amssymb}  
\usepackage{latexsym} 
\usepackage{comment}
\usepackage{color}
\usepackage{colonequals}
 \usepackage{epsfig}
\usepackage{tikz}
\usepackage{tikz-cd}
\usepackage{dsfont}
\usepackage{caption}
\usepackage{array}
\usepackage{makecell}
\usepackage{booktabs}
\usepackage{float}

\def\doi#1{%
	\href{https://doi.org/#1}{doi:#1}}%

\DeclareMathOperator{\GL}{GL}

\DeclareMathOperator{\PSL}{PSL}

\DeclareMathOperator{\Sing}{Sing}

\begin{document}

\newtheorem{theorem}{Theorem}[section]
\newtheorem{lemma}[theorem]{Lemma}
\newtheorem{proposition}[theorem]{Proposition}
\newtheorem{corollary}[theorem]{Corollary}
\newtheorem*{theorem*}{Theorem}
\newtheorem*{mainthmA}{Main Theorem A}
\newtheorem*{mainthmB}{Main Theorem B}
\newtheorem*{maincorA}{Corollary A}

\theoremstyle{definition}
\newtheorem{remark}[theorem]{Remark}
\newtheorem{definition}[theorem]{Definition}
\newtheorem{example}[theorem]{Example}
\newtheorem{notation}[theorem]{Notation}
\newtheorem*{notation*}{Notation}

\newtheorem*{problem*}{Problem}
\newtheorem*{question*}{Question}
\newtheorem*{acknowledgements}{Acknowledgements}

\numberwithin{equation}{section}

\renewcommand{\subjclassname}{%
	\textup{2020} Mathematics Subject Classification}

\title[
]{A minimality criterion for surfaces of general type with applications to product-quotient surfaces}

\author{Federico Fallucca}
\address{Dipartimento di Matematica\\
	Universit`a degli Studi di Trento\\
	Via Sommarive 14, 38123 Trento, Italy}
\email{federico.fallucca@unitn.it}


\maketitle 

\begin{abstract}
		We provide a numerical criterion to find the curves on a surface of general type that are contracted to its minimal model. We then apply this approach to product-quotient surfaces, where the criterion determines whether all such contracted curves are contained in reducible fibers of the two natural fibrations.
\end{abstract}


\section{Introduction}
Finding smooth rational curves with self-intersection $-1$ on a complex algebraic surface of general type is a difficult problem. Studying their configurations relative to other negative curves is equally challenging, yet necessary to determine the sequence of blow-downs that yields the minimal model.

A classical approach to constructing the minimal model consists in finding an effective canonical $\mathbb{R}$-divisor $D$ on the surface $S$. Since $DE=K_S E = -1$ for every exceptional curve $E$, the support of $D$ contains all  $(-1)$-curves. Contracting an exceptional curve $E$ yields a new surface $S'$ and a new effective canonical $\mathbb{R}$-divisor $D'$. Iterating this process terminates in finitely many steps, producing the minimal model of $S$. 

The construction of effective canonical divisors has been studied by several authors, such as Hirzebruch, Van der Geer, Van de Ven, and Zagier \cite{HVdeVen79,VGVV97,HZ77,VanGeer79}, who were interested, among other things, in investigating the minimality of Hilbert modular surfaces, see \cite{VanGeer88}. Their main strategy consists in studying special configurations of curves on the surface and taking a positive linear combination $D$ of these curves satisfying $D^2 = D K_S = K_S^2 > 0$. Indeed, from Hodge Index Theorem,  $K_S$ and $D$ are numerically equivalent, see also \cite[Prop. 7.3]{VanGeer88}.  

In this paper, we present a numerical criterion to find $(-1)$-curves on a surface of general type and to construct its minimal model. Our criterion extends the classical approach by considering canonical $\mathbb{R}$-divisors $D = A - B$ that, instead of being effective, are expressed as the difference of two effective $\mathbb{R}$-divisors. For any exceptional $(-1)$-curve $E$ on $S$, we have
\[
A E = K_S E + B  E = B  E - 1.
\]
Consequently, if $B  E < 1$, then $E$ must be a component of the support of $A$. However, computing $B E$ assumes knowledge of the relative configuration of $E$ and the components of $B$, which is typically a difficult geometric problem.

To avoid this obstacle, we exploit the intersection constraints that a $(-1)$-curve and other curves must satisfy on a surface of general type, as established by Bombieri \cite{Bomb73}, see Proposition \ref{prop: diseg_intersezione}. Specifically, we prove that if a certain real value $\tau(B)$, depending solely on $B$ and not on $E$, is less than one, then $B E \leq \tau(B) < 1$. This implies that the support of $A$ contains all $(-1)$-curves $E$ on $S$. 
\\
Furthermore, blowing down a $(-1)$-curve $E$ occurring in $A$ yields a new surface $S'$ and a pair of effective $\mathbb{R}$-divisors $A'$ and $B' = \sum \lambda_i C_i'$. Since $\tau(B') \leq \tau(B) < 1$ (see Theorem \ref{thm: minimality_criterion}) then $A'$ contains all $(-1)$-curves of $S'$ by the same argument as above. Iterating this process produces the minimal model $S_{\min}$ in finitely many steps. 

Our first main result is the following, see Theorem \ref{thm: minimality_criterion}: 
\begin{mainthmA}
	Let $S$ be a surface of general type, and let $A$ and $B$ be effective $\mathbb{R}$-divisors on $S$ such that $K_S\equiv A-B$. 
	Write $B = \sum \lambda_i C_i$, where the $C_i$ are distinct irreducible curves and $\lambda_i \in \mathbb{R}_{\geq 0}$. If
	\begin{equation}\label{eq: minimality_criterion_intro}
		\tau(B) := \sum_{i} \lambda_i \max\{1, 2p_a(C_i) - C_i^2 - 3\} < 1,
	\end{equation}
	then the support of $A$ contains all curves contracted by the morphism onto the minimal model $ S \to S_{\min}$. In particular, it contains all $(-1)$-curves of $S$.
\end{mainthmA}

We then adapt this criterion to product-quotient surfaces of general type. A product-quotient surface is the minimal resolution of singularities $\rho \colon S \to X$ of the quotient $X := (C_1 \times C_2)/G$ of a product of two curves by the diagonal action of a finite group $G$, see \cite{Cat00, BP12, BP16}. In this setting, the choice of the effective $\mathbb{R}$-divisors $A$ and $B$ becomes clear, see Proposition \ref{prop: K_S_as_A-B}.

This setting leads to Theorem \ref{thm: min_Function_localization_-1_curves}, our second main result, which we present here in a shortened version. Let $i_j$ denote the number of points $\overline{z} \in C_j/G$ such that the fiber over $\overline{z}$ of the fibration $X \to C_j/G$ contains a singularity of $X$, for $j=1,2$. Let $\Delta_1 \subseteq \mathbb{R}^{i_1}$ and $\Delta_2 \subseteq \mathbb{R}^{i_2}$ be the standard simplices.
\begin{mainthmB}
	Let $\rho \colon S \to X = (C_1 \times C_2)/G$ be a product-quotient surface of general type, and let $F \colon \Delta_1 \times \Delta_2 \to \mathbb{R}_{\geq 0}$ be the continuous piecewise linear function associated with $S$ (see Theorem \ref{thm: min_Function_localization_-1_curves}). If 
	\[
	F(\overline{\lambda}, \overline{\mu}) < 1
	\]
	for some $(\overline{\lambda}, \overline{\mu}) \in \Delta_1 \times \Delta_2$, then all curves contracted by $S\to S_{\min}$ are contained in reducible fibres of the two natural fibrations $f_1 \colon S \to C_1/G$ and $f_2 \colon S \to C_2/G$.
\end{mainthmB}
We note that the function $F$ is a purely combinatorial tool. It is independent of the geometry of the product-quotient surface and depends solely on its \textit{combinatorial structure}, namely:
\begin{itemize}
	\item the signatures of the quotient maps $\pi_j \colon C_j \to C_j/G$;
	\item the basket of singularities of $X$; and
	\item how these singularities are distributed among the fibers of the natural fibrations $X \to C_1/G$ and $X \to C_2/G$.
\end{itemize}
Example \ref{exmp: example} illustrates how the main theorem and the construction of $F$ work.

Main Theorem B immediately implies Proposition \ref{prop: loc_-1_curves_from_signatures} and Corollary \ref{cor: improvement_Lemma_Roberto}, which represents a major improvement of \cite[Prop. 4.7(2)]{BP12}. When all non-canonical singularities of \(X\) are concentrated over a single point \(\overline{x} \in C_1/G\) and a single point \(\overline{y} \in C_2/G\), these statements allow us to find the curves contracted to the minimal model of \(S\) (if any) directly from the signatures of \(\pi_j \colon C_j \to C_j/G\). Hence, there is no need to construct the function \(F\), which serves only as an intermediate tool.

As an application of Theorem \ref{thm: min_Function_localization_-1_curves}, we discuss the minimality of the remaining cases in \cite[Table 21]{fedeLincei} for regular product-quotient surfaces with $\chi(\mathcal{O}_S)=4$. This yields Theorem \ref{thm: improvement_classif_pg3_q0}:

\begin{theorem*}
	Let $S$ be a regular product-quotient surface with $23 \leq K_S^2 \leq 32$ and $\chi(\mathcal{O}_S)=4$. Then $S$ is a surface of general type belonging to one of the families described in \cite[Tables 9--21]{fedeLincei}. Furthermore, except for the families $no.537$, these surfaces are minimal.
\end{theorem*}

The surfaces in the families $no.537$ in \cite[Table 21]{fedeLincei} are not minimal; their minimal models are studied in Theorem \ref{thm: MainThm_note_2}.

\vspace{0.2cm}
The \texttt{MAGMA} \cite{BCP97} script and the computational data required to discuss the minimality of the cases in \cite[Table 21]{fedeLincei} are available in the public repository:
\begin{center}
	\href{https://github.com}{\texttt{MAGMA} script and results repository}
\end{center}
The repository includes the calculations listed in Table \ref{tab: minimality_results}. The \texttt{MAGMA} script contains three main functions: the first extracts the combinatorial structure of $S$ from its defining generating vectors; the second constructs the function $F$ and returns $F$ together with its minimum on $\Delta_1 \times \Delta_2$; the third returns the genus and self-intersection of the central components of $S$ from its combinatorial structure, according to theorems \ref{thm:Serrano} and \ref{thm: dec_Fibres_self_int_and_genus}.
\vspace{0.2cm}

The paper is organized as follows. In Section \ref{sec: discrepancies}, we provide some preliminaries on the discrepancies of cyclic quotient singularities. We prove Lemma \ref{lem: sol_tridiagonal_system}, where we give a closed formula for the inverse of the intersection matrix of a Hirzebruch-Jung string of a cyclic quotient singularity. This result is of independent interest and should already be known, although we have been unable to locate it in the existing literature. In Section \ref{sec: criterion_minimality_surf_gen_type}, we state and prove the first main Theorem \ref{thm: minimality_criterion}. Section \ref{sec: fibre_structure} is devoted to the second main Theorem \ref{thm: min_Function_localization_-1_curves} and to the study of the fibers of the natural fibrations of a product-quotient surface. In Section \ref{sec: minimality_PQ_pg3_q0}, we discuss the minimality of the remaining surfaces in \cite[Table 21]{fedeLincei}. Finally, in Section \ref{sec: minimality_PQ_pg0_q0} we discuss the minimality of product-quotient surfaces of general type with $p_g=q=0$ and $K^2 \geq -1$.
\begin{notation*}
	We denote by $\equiv$ the numerical equivalence relation between divisors. 
	
	Given a collection of positive integers $b_1,\dots, b_l\geq 2$, we denote by $[b_1, \dots, b_l]$ the continued fraction 
	\[
	  [b_1, \dots, b_l]:=b_1-\frac{1}{b_2-\frac{1}{
	  		b_3-\dots}}.
	\]
	Thus, $[b_1, \dots, b_l]$ is a fraction $\frac{n}{a}$, for some $1\leq a < n$ with  $\gcd(a,n)=1$. We denote by $D[b_1, \dots, b_l]:=a$ and by $N[b_1, \dots, b_l]:=n$ the denominator and numerator of the continued fraction $[b_1, \dots, b_l]$, respectively. 
	
\end{notation*}
\section{The discrepancies of a cyclic quotient singularity}\label{sec: discrepancies}
Let $X$ be a $\mathbb Q$-Gorenstein complex variety of dimension two. We denote by $K_X$ the class of the canonical Weil divisor on $X$, namely a Weil divisor such that $\mathcal O_{X}(K_X)\cong i_*(\Omega^n_{X^\circ})$, where $i\colon X^\circ \to X$ is the inclusion of the smooth locus of $X$. 

Note that a multiple of $K_X$ is a Cartier divisor. We recall the following 
\begin{definition}
	The Index of a point $p\in X$ is the minimal positive integer $I_p$ such that $I_pK_X$ is Cartier in a neighbourhood of $p$. 
	
	The Index $I$ is the minimal positive integer such that $IK_X$ is Cartier. In particular, $I=\text{lcm}\{I_p \colon p \in \Sing(X)\}$.  It is well-known \cite[Thm. 4-6-20]{Mat2002} that the index of a cyclic quotient singularity $p$ of type $\frac{1}{n}(1,a)$ is 
	\[
	  I_p=\frac{n}{\gcd(n,a+1)}. 
	\]
\end{definition}
Let $\rho\colon S\to X$ be the minimal resolution of singularities of $X$, and let $\{E_i\}$ be the family of all exceptional prime divisors of $\rho$. Then 
\begin{equation}\label{eq: discrepancies}
	IK_S=\rho^*(IK_X)+\sum_i \alpha_i E_i,  \qquad \makebox{with} \quad \alpha_i\in \mathbb Z. 
\end{equation}
Given a singular point $p\in X$ and $\{E_i\}$ be the exceptional prime divisors of the resolution of $p$, the coefficients $a_i:=\alpha_i/I$ of $E_i$ appearing in \eqref{eq: discrepancies} are called the \textit{discrepancies} of $p$. From the well-known Negativity Lemma in dimension two \cite[Lemma 1-8-10]{Mat2002}, all the discrepancies are non-positive, namely $a_i\leq 0$.  The point  $p$ is called a \textit{canonical} singularity when all the coefficients $a_i$ of $E_i$ are zero.   

In this work, we only focus on surfaces with at most isolated cyclic quotient  singularities. 
The minimal resolution of a cyclic quotient singularity $p\in X$ of type $\frac{1}{n}(1,a)$ consists of replacing $p$ by a Hirzebruch-Jung string $E_1\cup \dots \cup E_\ell$, where $E_i$ are smooth rational curves with $E_iE_{i+1}=1$, for $i=1,\dots, \ell-1$, and $E_iE_j=0$ otherwise. Their self-intersection is $E_i^2=-b_i$, where $b_i$ are the coefficients of the continued fraction 
\[\frac{n}{a}=b_1-\frac{1}{b_2-\frac{1}{
		b_3-\dots}}=[b_1,\dots, b_\ell].\]
	\begin{definition}
		Given a tuple of positive integers $(b_1,\dots, b_\ell)$ with $b_i\geq 2 $, the tridiagonal matrix
		\[
		A(b_1, \dots, b_\ell):=\begin{pmatrix}
			-b_1 & 1 & 0 & \dots & 0 \\
			1 & -b_{2} & 1 & \dots & 0 \\
			0 & 1 & -b_{3} & \ddots & \vdots \\
			\vdots & \vdots & \ddots & \ddots & 1 \\
			0 & 0 & \dots &1  & -b_\ell
		\end{pmatrix}
		\]
		 is called Hirzebruch-Jung matrix of weights $(b_1,\dots, b_\ell)$. Clearly, the intersection pair matrix of a Hirzebruch-Jung string $E_1\cup\dots \cup E_\ell$ of type $\frac{1}{n}(1,a)$, with $\frac{n}{a}=[b_1,\dots,b_\ell]$, is the Hirzebruch-Jung matrix of weights $(b_1,\dots, b_\ell)$. 
	\end{definition} 
\begin{remark}\label{rem: det_HJ_matrix}
		The determinant of a Hirzebruch-Jung matrix $A(b_1,\dots, b_\ell)$ is equal to $(-1)^\ell N[b_1,\dots, b_\ell]$, see \cite[Lem. 4.2.9]{Fra12}. 
\end{remark}
		 We provide a closed formula for the inverse of a Hirzebruch-Jung matrix. 
	\begin{lemma}\label{lem: sol_tridiagonal_system}
		The inverse of a Hirzebruch-Jung matrix $A(b_1, \dots, b_\ell)$ is 
		\[
		A^{-1}(b_1,\dots, b_\ell)=
		\left(-\frac{D[b_{\min(i,j)},\dots, b_{1}]\cdot D[b_{\max(i,j)},\dots, b_\ell]}{N[b_1,\dots, b_\ell]}\right)_{1\leq i,j\leq \ell}.
		\]
	\end{lemma}
	\begin{proof} We denote for short $A:=A(b_1,\dots, b_\ell)$. 
		
		The matrix $A$ is symmetric, so the $(i,j)$ and $(j,i)$ entries of the inverse of $A$ are equal to $(-1)^{i+j}\frac{\det(A_{ij})}{\det(A)}$, where $A_{ij}$ is the matrix obtained by removing the $i$-th row and $j$-th column from $A$, and $\det(A)=(-1)^lN[b_1,\dots, b_\ell]$ from Remark \ref{rem: det_HJ_matrix}. Thus, it is sufficient to determine $\det(A_{ij})$ for any $i\leq j$. Assume that $i=j$; then $A_{ii}$ consists of two blocks $A(b_1,\dots, b_{i-1})$ and $A(b_{i+1}, \dots, b_\ell)$, so that
		\begin{multline*}
			\det(A_{ii})=\det(A(b_1,\dots, b_{i-1}))\cdot \det(A(b_{i+1}, \dots, b_\ell))= \\ =(-1)^{\ell-1}N[b_1,\dots, b_{i-1}]\cdot N[b_{i+1},\dots, b_\ell].
		\end{multline*}
		Instead, assume that $i<j$; then $A_{ij}$ has the following structure 
		\[
		{\small A_{ij} = \left(
		\begin{array}{c|ccccc|c}
			A(b_1,\dots, b_{i-1}) & \mathbf{*} & \mathbf{*} & \mathbf{*}& \mathbf{*}& \mathbf{*} & \mathbf{0}\\ \hline 
			\mathbf{0} & 1 & -b_{i+1} & 1 & \dots & 0 & \mathbf{*}\\
			\mathbf{0} & 0 & 1 & -b_{i+2} & \ddots & \vdots & \mathbf{*} \\
			\vdots & \vdots & \ddots & \ddots & \ddots & 1 & \mathbf{*}\\
			\mathbf{0} & 0 & \dots & 0 & 1 & -b_{j-1} & \mathbf{*}\\
			\mathbf{0} & 0 & \dots & 0 & 0 & 1 &\mathbf{*} \\ \hline
			\mathbf{0} & & & \mathbf{0} & & & A(b_{j+1},\dots, b_\ell)
		\end{array}
		\right)}
		\]
		so
		\[
		\det(A_{ij})=(-1)^{l-(j-i)-1}N[b_1,\dots, b_{i-1}]\cdot N[b_{j+1}, \dots, b_\ell].
		\]	Finally, the thesis follows once one observes
		\[
		[b_{j},\dots, b_l]=b_{j}-\frac{1}{[b_{j+1},\dots, b_\ell]}
		=\frac{b_{j}N[b_{j+1},\dots, b_\ell]-D[b_{j+1},\dots, b_\ell]}{N[b_{j+1},\dots, b_\ell]} 
		\]
		so that $N[b_{j+1}, \dots, b_\ell]=D[b_{j}, \dots, b_\ell]$ and $N[b_1,\dots, b_{i-1}]=N[b_{i-1},\dots, b_1]=D[b_i,b_{i-1},\dots, b_1]$.
	\end{proof}
The discrepancies $a_i$ of $E_i$ of a cyclic quotient singularity 
can be computed using the 
partial continued fractions of $[b_1,\dots, b_\ell]$. This result is mentioned  in \cite[page 286]{HTU17}. We present here a short proof using Lemma \ref{lem: sol_tridiagonal_system}.
	\begin{proposition}
		\label{prop: formula_discrepancies}
		Let  $p\in X$ be a cyclic quotient singularity of type $\frac{1}{n}(1,a)$, with $\frac{n}{a}=[b_1,\dots, b_\ell]$, and let $E_1\cup \dots \cup E_\ell$ be its Hirzebruch-Jung string. 
		Then in a neighbourhood of the singular point of $p$ we have $K_S=\rho^*(K_X)+\sum_{i=1}^\ell a_iE_i$, where 
		\begin{equation*}
			a_i=-\left(1-\frac{D[b_i, \dots, b_\ell]+ D[b_i,b_{i-1},\dots, b_1]}{n}\right), \qquad i=1,\dots, \ell. 
		\end{equation*}
	\end{proposition}
\begin{proof}
	From \cite[Sec. 6.1]{Bar99} and \cite{Hir53}, then the vector $\bar{a}=\left(a_1,\dots, a_\ell\right)$ of the discrepancies of $p$  is equal to $\bar{a}=\frac{1}{n}\left(\bar{\lambda}+\bar{\mu}\right)-\sum_{i=1}^\ell v_i$, where $v_1,\dots, v_\ell$ is the canonical basis of $\mathbb R^\ell$, while $\bar{\lambda}$ and $\bar{\mu}$ are the solutions of the linear systems $A\bar{\lambda}=-nv_1$ and $A\bar{\mu}=-nv_\ell$.  Thus, $\bar{\lambda}=-nA^{-1}v_1$ and $\bar{\mu}=-nA^{-1}v_\ell$, so the thesis follows directly from Lemma \ref{lem: sol_tridiagonal_system}. 
\end{proof}
\begin{remark}
	One automatically obtains the well-known fact that a cyclic quotient singularity is a canonical singularity if and only if it is a rational double point, namely a singularity of type $\frac{1}{n}(1,n-1)$, $n\geq 2$. Indeed, if $p\in X$ of type $\frac{1}{n}(1,a)$ is a canonical singularity, then $0=a_1=-\left(1-\frac{a+1}{n}\right)$, which gives $a=n-1$. 	Conversely, if $p$ is a rational double point of type $\frac{1}{n}(1,n-1)$, then $\frac{n}{n-1}=[2, \dots, 2]$ with length $n-1$, then $a_i=-\left(1-\frac{(n-1-i+1)+i}{n}\right)=0$ for any $i=1,\dots, n-1$.
\end{remark}
\section{A criterion for the minimality of surfaces of general type}\label{sec: criterion_minimality_surf_gen_type}
In this section, we give a numerical criterion to establish the minimality of a surface of general type. More precisely, we provide a method to ensure that any curve contracted to the minimal model of the surface is contained in the support of a certain effective $\mathbb R$-divisor on the surface.

First of all we prove the following proposition, see also \cite[Prop. 4.4]{BP12},\cite[Prop. 4.3]{FP15}, and \cite[Cor. 6.8]{BP16}.
\begin{proposition}\label{prop: diseg_intersezione}
	Let $S$ be a surface of general type and let $E$ be a $(-1)$-curve of $S$. Then  any irreducible curve $C$ of $S$ satisfies 
	\[
	CE\leq \max\{1, 2p_a(C)-C^2-3\},
	\]
	and if $C$ is a smooth irreducible curve with $CE=1$, then $C^2\leq -2$. In other words, two $(-1)$-curves can not intersect. 
\end{proposition}
\begin{proof}
	Assume that $CE\geq 2$. Let $b\colon S\to S'$ be the blow-down given by the contraction of $E$. Then $b(C)$ becomes singular and by \cite[Rem. 4.3]{BP12}, then $K_{S'}b(C)>0$. Thus, we have 
	\[
	\begin{split}
		0<K_{S'}b(C) & =(K_S-E) (C+(C E)E) \\
		& =K_S C-C E-E C+CE \\
		&= 2p_a(C)-2-C^2-CE,
	\end{split}
	\]
	where the latter holds as $2p_a(C)-2=K_S\cdot C+C^2$. 
	
Now, let us suppose that $C\cdot E=1$ with $C$ smooth and rational. Then, after contracting $E$, $b(C)$ would still be a smooth rational curve and from \cite[Prop. 2.3]{BPHW2004}, then $b(C)^2<0$. Thus, we have 
	\[
	   0>b(C)^2=(C+E)\cdot (C+E)=C^2+1,
	\]
	so $C^2\leq -2$. 
\end{proof}
The first main theorem of the paper is the following 
\begin{theorem}\label{thm: minimality_criterion}
Let $S$ be a surface of general type and assume that the numerical canonical class $K_S$ of $S$ can be written as $K_S\equiv A-B$, where $A$ and $B$ are effective $\mathbb R$-divisors. 
 Write \(B=\sum \lambda_i C_i\), where the \(C_{i}\) are distinct irreducible curves and \(\lambda_i \in \mathbb{R}_{\geq 0}\). If
 \begin{equation}\label{eq: minimality_criterion}
 	\tau(B):=\sum_{i}\lambda_i \cdot \max\{1, 2p_a(C_i)-C_i^2-3\}<1,
 \end{equation}
 then $A$ contains all curves contracted by the morphism onto the minimal model $S\to S_{\min}$. In particular, it contains all $(-1)$-curves of $S$. 
\end{theorem}
\begin{proof}
	Let us consider a $(-1)$-curve $E$ of $S$. Then $K_SE=-1$ and from Proposition \ref{prop: diseg_intersezione} we have 
	\[
	  A E= (K_S+B) E= -1+ \sum_i \lambda_i C_iE \leq -1+\sum_i\lambda_i \max\{1, 2p_a(C_i)-C_i^2-3\}<0.
	\]
	Thus, $E$ is a component of $A$. 
	Blowing down \(E\) yields a new surface \(S^{\prime }\) and a new pair of effective \(\mathbb{R}\)-divisors \(A^{\prime }\) and \(B'=\sum \lambda_iC_i'\). We observe that
	\[K_{S^{\prime }}C_{i}^{\prime }-1=K_{S}C_{i}-C_{i}E-1\le  K_SC_i-1
	\]
	for any irreducible curve \(C_{i}'\) in \(B'\). Consequently, \(\tau(B')\leq \tau(B)<1\), which implies that \(A^{\prime }\) contains all \((-1)\)-curves of \(S^{\prime }\) by the same argument as above. By repeating this process, we obtain the minimal model of \(S\) in a finite number of steps.
\end{proof}
The rest of the paper is devoted to adapting Theorem \ref{thm: minimality_criterion} to product-quotient surfaces of general type, where the choice of the effective divisors $A$ and $B$ becomes clear, see Proposition \ref{prop: K_S_as_A-B} and Theorem \ref{thm: min_Function_localization_-1_curves}. First, we recall the decomposition of the singular fibers of the two natural fibrations of such surfaces.
\section{Fibre structure of a standard isotrivial fibration}\label{sec: fibre_structure}
Let us consider a finite group $G$ acting faithfully on two smooth projective curves $C_1$ and 
and $C_2$ having genus at least $2$. We consider the diagonal action of $G$ on the product $C_1 \times C_2$. 

Thus, $X:=(C_1\times C_2)/G$ may have at most cyclic quotient singularities arising from points of $C_1\times C_2$ having no-trivial stabilizer. 
\begin{definition}
	The minimal resolution of singularities $\rho\colon S\to X$ of $X$ is called \textit{product-quotient} surface of the \textit{quotient model} $X$. 
\end{definition}

Any product-quotient surface $S$ inherits from $X$ two fibrations\footnote{By abuse of notation, we still denote by $f_j\colon X\to C_j/G$ the two natural fibrations of $X$.}
\[
 f_1\colon S\to C_1/G \qquad \makebox{and} \qquad f_2\colon S\to C_2/G,
\] 
whose general fibres we denote respectively by $F_1$ and $F_2$. 
We are interested to study the structure of the singular fibres in further detail. Thus, we need 
\begin{theorem}(\cite[Thm. 2.1]{Ser96})\label{thm:Serrano}
	Let $\rho \colon S \to X=(C_1 \times C_2)/G$ be a product-quotient surface.
	Consider the natural fibration $f_2 \colon S \to C_2/G$.
	Take any point over $\bar{y} \in C_2/G$ and let $F$ denote the schematic fibre of $f_2$ over $\bar{y}$. Then
	\begin{itemize}
		\item[$\boldsymbol{(i)}$] The reduced structure of $F$ is the union of a smooth
		irreducible curve $Y$, called the central component of $F$,
		and either none or at least two mutually disjoint \emph{Hirzebruch-Jung strings}, each
		meeting $Y$ at one point, and each being contracted by $\rho$ to a singular point of $X$.
		These strings are in one-to-one
		correspondence with the branch points of $C_1 \to C_1/H$, where $H
		\subseteq G$ is the stabilizer of $y$, and if one string resolves a singularity of type $\frac{1}{n}(1,\,a)$, then the corresponding branch point has ramification index $n$.
		\item[$\boldsymbol{(ii)}$] The intersection of a string with $Y$ is transversal,
		and it takes place at only one of the end components of the string.
		\item[$\boldsymbol{(iii)}$] $Y$ is isomorphic to $C_1/H$, and has multiplicity
		equal to $|H|$ in $F$.
	\end{itemize}
	An analogous statement holds if one considers the other 
	$f_1 \colon S \to C_1/G$.
\end{theorem}
In particular, we have
\begin{equation}
	F_1 \simeq C_2, \quad  F_2 \simeq C_1,  \quad \textrm{and} \quad  F_1 F_2 = |G|. 
\end{equation}
Hence both fibrations $f_1$ and $f_2$ of $S$ are \emph{isotrivial}.

We denote by $F_2^{\overline{y}}:=f_2^*(\bar{y})$ the fibre of $S$ over a point $\bar{y}\in C_2/G$, and by $Y_2^{\bar{x}}$ the central component of $F_2^{\bar{y}}$, according to Theorem \ref{thm:Serrano}. Similarly,  $F_1^{\overline{x}}:=f_1^*(\bar{x})$  denotes the fibre of $S$ over a point $\bar{x}\in C_1/G$ and $Y_1^{\bar{x}}$ is its central component.   
%
%
 \begin{definition}\label{defn: Theta_i}\cite[Def. 1.12]{BP12}
 	Given a Galois cover $\pi\colon C\to C/G$ with Galois group $G$, we denote by \(e_{\overline{z}}\) the ramification index of any point in the fibre \(\pi^{-1}(\bar{z})\). Recall that this index coincides with the order of the stabilizer of any point in the corresponding fibre. 
 	\\
 	The signature of $\pi$ is the list of the genus of the base curve $C/G$ and the increasing sequence of the ramification indices of $\pi$, namely  $t:=(g(C/G)\vert e_{\overline{z_1}},\dots, e_{\overline{z_r}})$ with 
 	\[
 	e_{\overline{z_1}}\leq \dots \leq e_{\overline{z_r}}.
 	\]
 	When the genus of the base curve is zero, namely $C/G\cong \mathbb P^1$, we simply write $t=(e_{\overline{z_1}},\dots, e_{\overline{z_r}})$. 
    \\
	We also define
	\[
	\Theta=\Theta(t):=2g(C/G)-2+\sum_{i=1}^r
\left(1-\frac{1}{e_{\overline{z_i}}}\right).
\]
\end{definition}
 
Thus, for a point \(\overline{x}\in C_1/G\), we denote by \(e_{\overline{x}}\) the ramification index of the quotient map \(\pi_1\colon C_1\to C_1/G\), and similarly for a point \(\overline{y}\in C_2/G\). Furthermore, for brevity, we set \(\Theta_1:=\Theta(t_1)\) and \(\Theta_2:=\Theta(t_2)\), where \(t_{1}\) and \(t_{2}\) are the signatures of \(\pi _{1}\) and \(\pi _{2}\), respectively.
\begin{proposition}\label{prop:  coeffcients_F_i}
	Locally around  a singular point $\overline{(x,y)}\in X$ of type $\frac{1}{n}(1,a)$, with $\frac{n}{a}=[b_1,\dots, b_\ell]$, the fibres $F_1^{\overline{x}}$  and $F_2^{\overline{y}}$ of $S$ decomposes as follows
	\begin{equation*}\label{eq: Serrano_dec}
		F_1^{\overline{x}}= e_{\overline{x}}Y_1^{\overline{x}}+\sum_{i=1}^\ell e_{\overline{x}}\frac{D[b_i, \dots, b_\ell]}{n} E_i, \ \  \makebox{and} \quad F_2^{\overline{y}}= e_{\overline{y}}Y_2^{\overline{y}}+ \sum_{i=1}^\ell e_{\overline{y}}\frac{D[b_i, b_{i-1}, \dots, b_1]}{n} E_i,
	\end{equation*}
where $E_1\cup \dots\cup  E_\ell$ is the Hirzebruch-Jung string  of $\overline{(x,y)}$.
\end{proposition}
\begin{proof}
	By Theorem \ref{thm:Serrano}$(i)$ and $(iii)$, then we can write $F_1^{\overline{x}}= e_{\overline{x}}Y_1^{\overline{x}}+\sum_{i=1}^\ell\mu_i E_i$ and $ F_2^{\overline{y}}= e_{\overline{y}}Y_2^{\overline{y}}+ \sum_{i=1}^\ell \lambda_i E_i$, so we only need to compute the vectors $\bar{\mu}$ and $\bar{\lambda}$. We determine $\bar{\lambda}$, in a similar way one obtains also $\bar{\mu}$. 
	From Theorem \ref{thm:Serrano}$(i)$ and $(ii)$, then $Y_2^{\overline{y}}E_\ell=1$ and $Y_2^{\overline{y}}E_j=0$ otherwise, for $j=1, \dots, \ell-1$. This permits to compute $\bar{\lambda}$ via intersection theory. Indeed, we have 
	\[
	F_2^{\overline{y}}E_j=0=e_{\overline{y}}Y_2^{\bar{y}}E_j + \left( \sum_{i=1}^\ell \lambda_i E_i\right)E_j, \qquad j=1,\dots, \ell,
	\]
	which is equivalent to solve the linear system $A\bar{\lambda}=-e_{\bar{y}}v_{\ell}$, where $A$ is the Hirzebruch-Jung matrix of weights $(b_1,\dots, b_\ell)$ and $v_1,\dots, v_\ell$ is the canonical basis of $\mathbb R^\ell$. Thus, $\bar{\lambda}=-e_{\bar{y}}A^{-1}v_{\ell}$ and the thesis follows directly from Lemma \ref{lem: sol_tridiagonal_system}. 
\end{proof}
Proposition \ref{prop:  coeffcients_F_i} justifies the following
\\
	{\bf Notation.} Given a cyclic quotient singularity \(p\in \text{Sing}(X)\) of type \(\frac{1}{n_p}(1, a_p)\), we denote by \([b_{1}^p, \dots, b_{\ell_p}^p]\) the Hirzebruch-Jung continued fraction of \(\frac{n_{p}}{a_{p}}\). We denote by $D_{i,\ell_p}^p$ the denominator of the partial continued fraction \([b_{i}^p, b_{i+1}^p, \dots, b_{\ell_p}^p]\) and by \(D_{i,1}^p\) the denominator of \([b_{i}^p, b_{i-1}^p\dots, b_{1}^p]\). Finally, the Hirzebruch-Jung string arising from the resolution of \(p\) is denoted by \(E_{1}^p \cup \dots \cup E_{\ell_p}^p\).
\begin{theorem}\label{thm: dec_Fibres_self_int_and_genus}
	Let us consider a singular fibre $F_1^{\overline{x}}$ of the isotrivial fibration $f_1\colon S\to C_1/G$ over a point $\overline{x}\in C_1/G$ and its central component $Y_1^{\overline{x}}$. Then:
	\[
		\begin{split}
		F_1^{\overline{x}} & =e_{\overline{x}}Y_1^{\overline{x}}+ \sum_{p\in \Sing(X)\cap f_1^{-1}(\overline{x})}\left(\sum_{i=1}^{\ell_p} e_{\overline{x}}\frac{D_{i, \ell_p}^p}{n_p}E_{i}^p\right), \\ 
		(Y_1^{\overline{x}})^2 & =\sum_{p\in \Sing(X)\cap f_1^{-1}(\overline{x})}-\frac{a_p}{n_p}\in \mathbb Z, \qquad \makebox{and} \\
				2g(C_2)-2 & =e_{\overline{x}}\left(2g(Y_1^{\overline{x}})-2+\sum_{p\in \Sing(X)\cap f_1^{-1}(\overline{x})}\left(1-\frac{1}{n_p}\right)\right).
	\end{split}
	\]
	Analogously, given a singular fibre $F_2^{\overline{y}}$ over $\overline{y}\in C_2/G$, then 
	\[
	\begin{split}
		F_2^{\overline{y}} & =e_{\overline{y}}Y_2^{\overline{y}}+ \sum_{p\in \Sing(X)\cap f_2^{-1}(\overline{y})} \left(\sum_{i=1}^{\ell_p}e_{\overline{y}}\frac{D_{i, 1}^p}{n_p} E_{i}^p\right), \\ 
		(Y_2^{\overline{y}})^2 & =\sum_{p\in \Sing(X)\cap f_2^{-1}(\overline{y})}-\frac{a'_p}{n_p}\in \mathbb Z, \qquad \makebox{and} \\
		2g(C_1)-2 & =e_{\overline{y}}\left(2g(Y_2^{\overline{y}})-2+\sum_{p\in \Sing(X)\cap f_2^{-1}(\overline{y})}\left(1-\frac{1}{n_p}\right)\right),
	\end{split}
	\]
	where $1\leq a_p'\leq n_p-1$ denotes the multiplicative inverse of $a_p$ in $\left(\mathbb Z/n_p\right)^*.$
	Finally, the intersection between two central components $Y_1^{\overline{x}}$ and $Y_2^{\overline{y}}$ is equal to 
	\[
	 Y_1^{\overline{x}}Y_2^{\overline{y}}=\frac{\vert G\vert}{e_{\overline{x}}e_{\overline{y}}}-\sum_{p\in \Sing(X)\cap f_1^{-1}(\overline{x})\cap f_2^{-1}(\overline{y})}\frac{1}{n_p}.
	\]
\end{theorem}
\begin{proof}
	The decomposition of the singular fibres in irreducible components 
	follows directly from Proposition \ref{prop: coeffcients_F_i} applied iteratively for any singular point of $X$ belonging to such fibre. The self-intersection of the central components is given by \cite[Prop. 2.8]{Po10}, while the genus formulae are a consequence of Theorem \ref{thm:Serrano}$(i)$ and $(iii)$. It only remains to compute the intersection between $Y_1^{\overline{x}}$ and $Y_2^{\overline{y}}$. Consider a singular point point $p$ of $X$ belonging to the fibre of $\overline{x}$. If $p$ is not contained in the fibre of $\overline{y}$, then $Y_2^{\overline{y}}$ would intersect none of the components of the Hirzebruch-Jung strings of $p$. Conversely, if $p$ is contained in the fibre of $\overline{y}$ we would have 
	$E_{\ell_p}^pY_2^{\overline{y}}=1$ and $E_{i}^pY_2^{\overline{y}}=0$ otherwise, $i=1,\dots, \ell_p-1$. Thus, decomposing $F_1^{\overline{x}}$ in irreducible components and intersecting  $F_1^{\overline{x}}$ with $Y_2^{\overline{y}}$ we obtain 
	\[
	\frac{\vert G\vert}{e_{\overline{y}}}=F_1^{\overline{x}}Y_2^{\overline{y}}=e_{\overline{x}}Y_1^{\overline{x}}Y_2^{\overline{y}}+\sum_{p\in \Sing(X)\cap f_1^{-1}(\overline{x})\cap f_2^{-1}(\overline{y})} e_{\overline{x}}\frac{1}{n_p}E_{\ell_p}^pY_2^{\overline{y}},
	\]
	from which we automatically deduce the formula for $Y_1^{\overline{x}}Y_2^{\overline{y}}$.
\end{proof}
\begin{definition}
	For any $j=1,2$, let $I_j$ be the set of points $\overline{z}\in C_j/G$ such that $\Sing(X)\cap f_j^{-1}(\overline{z})$ is nonempty. 
	We consider the standard simplices 
	\[
	\Delta_1:=\left\{(\lambda_{\overline{x}})_{\overline{x}\in I_1}\colon \sum\lambda_{\overline{x}}=1, \lambda_{\overline{x}}\geq 0\right\} \  \text{and} \ \ \Delta_2:=\left\{(\mu_{\overline{y}})_{\overline{y}\in I_2}\colon \sum\mu_{\overline{y}}=1, \mu_{\overline{y}}\geq 0\right\}.
	\]
%
\end{definition}
\begin{remark}
	We observe that for any singular point $\overline{(x,y)}\in X$, 
	then  $\overline{x}\in I_1$, $\overline{y}\in I_2$, and so 
 we can consider the corresponding variables $\lambda_{\overline{x}}$ and $\mu_{\overline{y}}$.
\end{remark}
\begin{proposition}\label{prop: K_S_as_A-B}
	Let \(\rho\colon S\to X=(C_1\times C_2)/G\) be a product-quotient surface. 
	Let us fix \((\lambda_{\overline{x}})_{\overline{x}\in I_1}\in \Delta_1\) and  \((\mu_{\overline{y}})_{\overline{y}\in I_2}\in \Delta_2\). 
	
	The numerical canonical class \(K_{S}\) can be expressed as the difference of two effective $\mathbb R$-divisors $A$ and $B$, \(K_S \equiv_{num} A - B\), both depending on the chosen pair $(\lambda_{\overline{x}})_{\overline{x}\in I_1}$ and $(\mu_{\overline{y}})_{\overline{y}\in I_2}$. More precisely, we have
	\begin{equation*}A := \sum_{\overline{x}\in I_1} \delta_{\overline{x}} \lambda_{\overline{x}} Y_1^{\overline{x}} + \sum_{\overline{y}\in I_2} \delta_{\overline{y}} \mu_{\overline{y}} Y_2^{\overline{y}} + \sum_{p\in \Sing(X)} \left(\sum_{i=1}^{\ell_p} \max\{\xi_{i,p}, 0\} E_{i}^p\right),
	\end{equation*}
	and
	\begin{equation*}
		B := \sum_{p\in \Sing(X)} \left(\sum_{i=1}^{\ell_p} \max\{-\xi_{i,p}, 0\} E_{i}^p\right).
	\end{equation*}
For each point \(p = \overline{(x,y)} \in \Sing(X)\), the coefficients \(\xi _{i,p}\) are linear functions of \(\lambda _{\overline{x}}\) and \(\mu _{\overline{y}}\) defined by:
\begin{equation}\label{eq: value_of_Xi_i,P}
	\xi _{i,p}(\lambda _{\overline{x}},\mu _{\overline{y}}):=\frac{D_{i, \ell_{p}}^p}{n_{p}}(\delta_{\overline{x}}\lambda _{\overline{x}}+1)+\frac{D_{i,1}^p}{n_{p}}(\delta_{\overline{y}}\mu _{\overline{y}}+1)-1,
\end{equation}
where $\delta_{\bar{x}}:=\Theta_1e_{\bar{x}}$, $\delta_{\bar{y}}:=\Theta_2e_{\bar{y}}$ and $\Theta_1,\Theta_2$ are those in Definition \ref{defn: Theta_i}. 
\end{proposition}
\begin{proof}
	Note that $\vert G\vert K_X$ is Cartier and $\rho^*(\vert G\vert K_X)$ is numerically equivalent to $(2g(C_1)-2) F_1+(2g(C_2)-2)F_2$. Thus, applying Proposition \ref{prop: formula_discrepancies} recursively for any singular point of $X$ we obtain that $K_S$ is numerically equivalent to 
\begin{equation}\label{eq: K_S_for_prod_quo}
		K_S  \equiv \Theta_1F_1+\Theta_2F_2
		-\sum_{p\in \Sing(X)}\left(\sum_{i=1}^{\ell_p} \left(1-\frac{D_{i, \ell_p}^p+D_{i,1}^p}{n_p}\right)E_{i}^p\right).
\end{equation}
	Note that we replaced \((2g(C_i)-2)/|G|\) by \(\Theta _{i}\) from Riemann-Hurwitz formula applied to the covers \(\pi_i \colon C_i \to C_i/G\), for \(i=1,2\). 
	
	Now, for any fixed pair \((\lambda_{\overline{x}})_{\overline{x}\in I_1} \in \Delta_1\) and \((\mu_{\overline{y}})_{\overline{y}\in I_2} \in \Delta_2\), we can consider the following convex combinations of fibres up to numerical equivalence:
	\[F_{1}\equiv \sum _{\overline{x}\in I_{1}}\lambda _{\overline{x}}F_{1}^{\overline{x}}\quad \text{and}\quad F_{2}\equiv\sum _{\overline{y}\in I_{2}}\mu _{\overline{y}}F_{2}^{\overline{y}}.\]
	
	 Replacing $F_1^{\overline{x}}$ and $F_2^{\overline{y}}$ with their  decomposition given by Theorem \ref{thm: dec_Fibres_self_int_and_genus}, for any $\overline{x}\in I_1$ and $\overline{y}\in I_2$, we write 
	 \begin{equation}\label{eq: descrption_F_1_and_F_2}
	 	\begin{split}
	 		\Theta_1F_1+\Theta_2F_2 & \equiv\sum_{\overline{x}\in I_1}\delta_{\overline{x}}\lambda_{\overline{x}}Y_1^{\overline{x}}+\sum_{\overline{y}\in I_2}\delta_{\overline{y}}\mu_{\overline{y}}Y_2^{\overline{y}}+\\ & + \sum_{\overline{x}\in I_1}\lambda_{\overline{x}}\left(\sum_{p\in\Sing(X)\cap f_1^{-1}(\overline{x})}\left(\sum_{i=1}^{\ell_p}\delta_{\overline{x}}\frac{D_{i, \ell_p}^p}{n_p}E_{i}^p\right)\right)+ \\ & +
	 		\sum_{\overline{y}\in I_2}\mu_{\overline{y}}\left(\sum_{p\in\Sing(X)\cap f_2^{-1}(\overline{y})}\left(\sum_{i=1}^{\ell_p}\delta_{\overline{y}}\frac{D_{i, 1}^p}{n_p}E_{i}^p\right)\right).
	 	\end{split}
	 \end{equation}
 However, the last two terms of the right side can be rearranged as 
\begin{equation}\label{eq: rearrangement}
    \sum_{p=\overline{(x,y)}\in \Sing(X)} \left(\sum_{i=1}^{\ell_p}\left(\delta_{\overline{x}}\frac{D_{i, \ell_p}^p}{n_p}\lambda_{\overline{x}}+\delta_{\overline{y}}\frac{D_{i,1}^p}{n_p}\mu_{\overline{y}}\right)E_{i}^p\right), 
\end{equation}
Finally, replacing $\Theta_1F_1+\Theta_2F_2$ into the numerical canonical class $K_S$ in \eqref{eq: K_S_for_prod_quo} with the expression from Equation \eqref{eq: descrption_F_1_and_F_2} and using \eqref{eq: rearrangement}, we obtain 
\begin{equation}\label{eq: K_S_in_function_except_components}
	K_S\equiv \sum_{\overline{x}\in I_1}\delta_{\overline{x}}\lambda_{\overline{x}}Y_1^{\overline{x}}+\sum_{\overline{y}\in I_2}\delta_{\overline{y}}\mu_{\overline{y}}Y_2^{\overline{y}}  +\sum_{p\in\Sing(X)}\left(\sum_{i=1}^{\ell_p}\xi_{i,p}E_{i}^p\right).
\end{equation}
To write $K_{S}$ as the difference of the effective divisors $A$ and $B$, we must determine which exceptional components $E_{i}^p$ of the resolution of $X$ appear with a positive coefficient in Formula \eqref{eq: K_S_in_function_except_components}. By writing each $\xi_{i,p}$ as the difference $\xi_{i,p} = \max\{\xi_{i,p}, 0\} - \max\{-\xi_{i,p}, 0\}$, it follows that $E_{i}^p$ appears in $A$ with coefficient $\max\{\xi_{i,p}, 0\}$ and in $B$ with coefficient $\max\{-\xi_{i,p}, 0\}$.
\end{proof}
We can finally state and prove the second main theorem of the paper. 
\begin{theorem}\label{thm: min_Function_localization_-1_curves}
	Let $S$ be a product-quotient surface of general type with quotient model $X=(C_1\times C_2)/G$. Let $F\colon \Delta_1\times \Delta_2\to \mathbb  R_{\geq 0}$ be the continuous piecewise linear function 
	\[
		F((\lambda_{\overline{x}})_{\overline{x}\in I_1}, (\mu_{\overline{y}})_{\overline{y}\in I_2}):= \sum_{p\in \Sing(X)} \left( \sum_{i=1}^{\ell_p}\max\left\{-\xi_{i,p}, 0\right\}\max\left\{1, b_{i}^p-3\right\}\right),
\]
  where, for each $p=\overline{(x,y)}\in \Sing(X)$, the terms $\xi_{i,p}$ are the linear functions of $\lambda_{\overline{x}}$ and $\mu_{\overline{y}}$ defined in \eqref{eq: value_of_Xi_i,P}.	
  
  Let us suppose that it there exists a pair $((\lambda_{\overline{x}})_{\overline{x}\in I_1}, (\mu_{\overline{y}})_{\overline{y}\in I_2})\in\Delta_1\times \Delta_2$ such that 
	\[
	  F((\lambda_{\overline{x}})_{\overline{x}\in I_1}, (\mu_{\overline{y}})_{\overline{y}\in I_2}) < 1.
	\]
	Then all curves contracted to the minimal model of $S$ are contained in reducible fibres of \(f_1\colon S\to C_1/G\) and \(f_2\colon S\to C_2/G\).
\end{theorem}
\begin{proof}
	Fixing \((\lambda_{\overline{x}})_{\overline{x}\in I_1}\in\Delta_1\) and \((\mu_{\overline{y}})_{\overline{y}\in I_2}\in \Delta_2\), we can write \(K_S \equiv A-B\), according to  Proposition \ref{prop: K_S_as_A-B}. We observe that, for any choice of \((\lambda_{\overline{x}})\) and \((\mu_{\overline{y}})\), then the real value \(F((\lambda_{\overline{x}})_{\overline{x}\in I_1}, (\mu_{\overline{y}})_{\overline{y}\in I_2})\) coincides with the left-hand side $\tau(B)$ of Theorem \ref{thm: minimality_criterion}. By assumption, there exists a pair \(((\lambda_{\overline{x}}), (\mu_{\overline{y}})) \in \Delta_1 \times \Delta_2\) such that \(F((\lambda_{\overline{x}})_{\overline{x}\in I_1}, (\mu_{\overline{y}})_{\overline{y}\in I_2}) < 1\); thus, Theorem \ref{thm: minimality_criterion} applies. 
	However, $A$ consists only of some central components of \(S\) and exceptional components of the resolution of $X$, which are contained in reducible fibres of the two natural fibrations of $S$.
\end{proof}
\begin{remark}\label{rem: the_sum_is_not_affect_by_rational_double_points}
	The function \(F\) is a sum over the singular points of \(X\). However, singular points of \(X\) that are canonical singularities do not affect the sum; thus, the summation can be taken over the subset of singular points that are not canonical. Indeed, if $p=\overline{(x,y)}\in \Sing(X)$ is a canonical singularity of type \(\frac{1}{n}(1,n-1)\), then:
	\[
	\begin{split}
		\xi_{i,p}& =\frac{n-1-(i-1)}{n}(\delta_{\overline{x}}\lambda_{\overline{x}}+1)+\frac{i}{n}(\delta_{\overline{y}}\mu_{\overline{y}}+1)-1 \\ 
		&= \frac{n-i}{n}\delta_{\overline{x}}\lambda_{\overline{x}}+\frac{i}{n}(\delta_{\overline{y}}\mu_{\overline{y}})\geq 0.
	\end{split}
	\]
	Hence we would have $\max\{-\xi_{i,p},0\}=0$ independently by the choice of $\lambda_{\overline{x}}$ and $\mu_{\bar{y}}$. In other words, the contribution of the point $p$ to $F$ is zero. 
\end{remark}
\begin{remark}
	Assume that all the singularities of $X$ lie over one point of $C_1/G$ and over one point of $C_2/G$; in other words $I_1$ and $I_2$ consists of only one point, and so $\Delta_1$ and $\Delta_2$ are the standard  $0$-simplices. Hence 
	$F$ is a constant function, and for any $p=\overline{(x,y)}\in \Sing(X)$ the terms $\xi_{i,p}$ appearing in $F$ simply become: 
	\[
	\xi _{i,p}=\frac{D_{i, \ell_{p}}^p}{n_{p}}(\delta_{\overline{x}}+1)+\frac{D_{i,1}^p}{n_{p}}(\delta_{\overline{y}}+1)-1. 
	\]
	Thus, when the constant $F$ is less than one, then we automatically obtain that all curves contracted to the minimal model of $S$ are contained in reducible fibres of the two natural fibrations of $S$. 
\end{remark}
We provide an easy example to show how Theorem \ref{thm: min_Function_localization_-1_curves} works.
\begin{example}\label{exmp: example}
	Consider a product-quotient surface belonging to the family no.526 in \cite[Table 21]{fedeLincei}, see Table \ref{table: example_526}.
	\begin{table}[H]
		\centering
		\hspace*{-1.8cm}
		\renewcommand{\arraystretch}{1.2}
		\begin{tabular}{| c| c| c|c| c| c|  c| }
			\hline
			$no.$ & $G(d,n)$ & Sing($X$) &	 $t_1$ & $t_2$  & $\Theta_1$ & $\Theta_2$ \\
			\hline
			\hline
			$ 526 $ &$\langle 120, 35\rangle$ & \makecell[c]{$1/5, 1/3, 2/3,4/5$}  & $ 2, 6, 10 $ & $ 2^2, 3, 5 $ &   $\frac{7}{30}$ & $\frac{7}{15}$\\
		\hline
				\end{tabular}
		\caption{The information of the family of product-quotient surfaces no.526 with $p_g=3$, $q=0$.} \label{table: example_526}
	\end{table}
\vspace{-0.8cm}

   It is easy to verify that there are two points of $C_1/G$ containing some singular point of $X$; more precisely the point with ramification index $10$ contains two singular points of type  $\frac{1}{5}$ and $\frac{4}{5}$, while the point with ramification index $6$ contains two singular points of type $\frac{1}{3}$ and $\frac{2}{3}$. Thus we have two variables $(\lambda_1, \lambda_2)\in \Delta_1$.
   
   Similarly, we have two points of $C_2/G$ containing some singular point of $X$, that of ramification index $5$ containing $\frac{1}{5}$ and $\frac{4}{5}$, and that with ramification index $3$, containing $\frac{1}{3}$ and $\frac{2}{3}$. Thus, we also have in this case two variables $(\mu_1, \mu_2)\in \Delta_2$. 
   The function $F$ is then
   \[
      F(\lambda_1,\lambda_2,\mu_1,\mu_2)=2\max\left\{\frac{3}{5}-\frac{7}{15}(\lambda_1+\mu_1),0\right\}+\max\left\{\frac{1}{3}-\frac{7}{15}(\lambda_2+\mu_2),0\right\}
   \] 
   Choosing $(\lambda_1,\lambda_2)=(1,0)$ and $(\mu_1,\mu_2)=(\frac{2}{7},\frac{5}{7})$ we obtain $F=0$.  
   Then Theorem \ref{thm: min_Function_localization_-1_curves} holds and so any $(-1)$-curve must be one of the central components of $S$. However, one can apply Theorem \ref{thm: dec_Fibres_self_int_and_genus} and compute that the genus of the central components of the fibres over the points of ramification $10$ and $5$ have both genus $3$, while those of the points with ramification index $6$ and $3$ have both genus $5$. Thus, $S$ is a minimal surface of general type. 
\end{example}
\begin{remark}\label{rem: F_only_dep_from_num_data}
	As shown by the example, the function \(F\) depends only on the combinatorial structure of $S$, namely on the signatures \(t_{1}\) and \(t_{2}\) and on the basket of singularities of \(X\), by accounting for how the singularities are partitioned among the fibres of \(f_1\colon X \to C_1/G\) and \(f_2\colon X\to C_2/G\). 
\end{remark}
We conclude the section with some consequences of Theorem \ref{thm: min_Function_localization_-1_curves}, namely Proposition \ref{prop: loc_-1_curves_from_signatures} and Corollary \ref{cor: improvement_Lemma_Roberto}. Under the assumption that all non-canonical singularities of $X$ lie over a single point \(\overline{x}\in C_1/G\) and over a single point \(\overline{y}\in C_2/G\), this result allows us to to find the $(-1)$-curves on a product-quotient surface of general type directly from the signatures \(t_{j}\) of the Galois coverings \(\pi_j\colon C_j\to C_j/G\), without involving the function $F$.

We denote by $\widehat{t_1}$ the signature obtained from $t_1$ by removing the ramification index $e_{\overline{x}}$ of the point $\overline{x}\in C_1/G$. Similarly, $\widehat{t_2}$ is obtained from $t_2$ by removing $e_{\overline{y}}$. For instance, if we have a pair of signatures $t_1=(g_1\vert 2,3,5)$ and $t_2=(g_2\vert 2,6,6,7)$, and if all the non-canonical singularities lie over the two points of ramification index $3$ and $6$, then $\widehat{t_1}=(g_1\vert 2,5)$ and $\widehat{t_2}=(g_2\vert 2,6,7)$. Finally, we define 
\[
\widehat{\Theta_1}:=\Theta_1+\frac{1}{e_{\overline{x}}}, \qquad \makebox{and} \qquad \widehat{\Theta_2}:=\Theta_2+\frac{1}{e_{\overline{y}}}. 
\]
Note that $\widehat{\Theta_j}=\Theta(\widehat{t_j})+1$, for  $j=1,2$. 
\begin{lemma}\label{lem: technical_lemma_-1_curves_from_signatures}
	Let $S$ be a product-quotient surface of general type. Let us assume that all the non-canonical singularities of the quotient model $X=(C_1\times C_2)/G$ lie over a single point $\overline{x}\in C_1/G$ and over a single point $\overline{y}\in C_2/G$. If 
	\[
	   \widehat{\Theta_1}+\widehat{\Theta_2}\geq 1,
	\]
	then all curves contracted to the minimal model of $S$ are contained in reducible fibres of $f_1\colon S\to C_1/G$ and $f_2\colon S\to C_2/G$. 
\end{lemma} 
\begin{proof}
	From Remark \ref{rem: the_sum_is_not_affect_by_rational_double_points}, then $F$ is a function depending only on the two variables $\lambda_{\overline{x}}$ and $\mu_{\overline{y}}$.  We want to compute $F$ at $\lambda_{\overline{x}}=\mu_{\overline{y}}=1$ and prove that $F(1,1)=0$. 
	
	To do this, it is sufficient to study $\xi_{i,p}(1,1)$ for any singular point $p=\overline{(x,g\cdot y)}\in X$ lying over $\overline{x}$ and $\overline{y}$. Remembering that $n_p$ divides both $e_{\overline{x}}$ and $e_{\overline{y}}$, then
	\[
	\begin{split}
	  \xi_{i,p}(1,1) &=\frac{D_{i, \ell_{p}}^p}{n_p}(\Theta _{1}e_{\overline{x}}+1)+\frac{D_{i, 1}^p}{n_p}(\Theta _{2}e_{\overline{y}}+1)-1 \\ 
	  &\geq \Theta_1\frac{e_{\overline{x}}}{n_p}+\Theta_2\frac{e_{\overline{y}}}{n_p}+\frac{2}{n_p}-1\\
	  &\geq \Theta_1+\Theta_2+\frac{2}{n_p}-1 \\
	  & \geq \widehat{\Theta_1}+\widehat{\Theta_2}-1\\
	  & \geq 0.
  \end{split}
	\]
	Thus, $\max\{-\xi_{i,p}(1,1), 0\}=0$ for any singular point $p$ lying over $\overline{x}$ and $\overline{y}$. Hence $F(1,1)=0$ and Theorem \ref{thm: min_Function_localization_-1_curves} applies. 
\end{proof}
We consider the lexicographic order $\geq$ on the set of signatures with the same fixed genus $g\geq 0$. 
The length of a signature $(g\vert m_1\dots, m_r)$ is $r$. 
\begin{proposition}\label{prop: loc_-1_curves_from_signatures}
	Let $S$ be a product-quotient surface of general type such that all the non-canonical singularities of the quotient model $X=(C_1\times C_2)/G$ lie over a single point $\overline{x}\in C_1/G$ and over a single point $\overline{y}\in C_2/G$. Furthermore, let us assume that the pair of signatures $t_1$ and $t_2$ of $S$ satisfies one of the following conditions: 
	\begin{enumerate}
		\item one between $t_1$ or $t_2$ has at least length $5$; or 
		\item $t_1$ and $t_2$  have both length $4$; or 
		\item $t_1$ has length $4$ and $t_2$ has length $3$  and they satisfies one among 
		\begin{itemize}
			\item[$(a)$] $\widehat{t_1}\geq (g_1\vert 2,3,6)$;
			\item[$(b)$] $(g_1\vert 2,2,6)\leq \widehat{t_1}<(g_1\vert 2,3,6)$, $\widehat{t_2}\geq (g_2\vert 2,3)$; 
			\item[$(c)$] $ \widehat{t_1}=(g_1\vert 2,2,5),(g_1\vert 2,2,4)$, $\widehat{t_2}\geq (g_2\vert 2,4)$; 
			\item[$(d)$] $ \widehat{t_1}=(g_1\vert 2,2,3)$, $\widehat{t_2}\geq (g_2\vert 2,6)$; 
			\item[$(e)$] $\widehat{t_1}=(g_1\vert 2,2,2)$, $\widehat{t_2}\geq (g_2\vert 3,6)$;
		\end{itemize}
	or 
	\item $t_1$ and $t_2$ have both length $3$ and they satisfies $\frac{1}{a_1}+\frac{1}{b_1}+\frac{1}{a_2}+\frac{1}{b_2}\leq 1$, where $\widehat{t_1}=(g_1\vert a_1,b_1)$ and $\widehat{t_2}=(g_2\vert a_2,b_2)$. 
	\end{enumerate}
Then all curves contracted to the minimal model of $S$ are contained in reducible fibres of $f_1\colon S\to C_1/G$ and $f_2\colon S\to C_2/G$. 
\end{proposition}
\begin{proof}
	In the first case, assuming that $t_1$ has length at least $5$, one would have 
	\[
	\widehat{\Theta_1}+\widehat{\Theta_2}\geq \widehat{\Theta_1}\geq -2+4\left(1-\frac{1}{2}\right)+1\geq 1. 
	\]
	Instead, in the second case 
	\[
	\widehat{\Theta_1}+\widehat{\Theta_2}\geq 2\left(-2+3\left(1-\frac{1}{2}\right)+1\right)\geq 1. 
	\]
	One can easily verify that also the other remain cases implies $\widehat{\Theta_1}+\widehat{\Theta_2}\geq 1$. Thus, Lemma \ref{lem: technical_lemma_-1_curves_from_signatures} applies and the thesis follows. 
\end{proof}
We conclude with an improvement of \cite[Prop. 4.7(2)]{BP12}. 
\begin{corollary}\label{cor: improvement_Lemma_Roberto}
	Let $S$ be a product-quotient surface of general type such that the quotient model $X=(C_1\times C_2)/G$ has at most one singularity of type $\frac{1}{n}(1,a)$, and canonical singularities. Let us suppose that the signatures $t_1$ and $t_2$ of $S$ fall in one of the cases of Proposition \ref{prop: loc_-1_curves_from_signatures}. Then $S$ is minimal.
\end{corollary}
\begin{proof}
	If $X$ has at most canonical singularities, then $S$ is minimal, see \cite[Prop. 4.7(2)]{BP12}. Thus, let us assume that $\frac{1}{n}(1,a)$ is non-canonical. Then Proposition \ref{prop: loc_-1_curves_from_signatures} applies and a $(-1)$-curve $E$ would be one of the central components of $S$. However, $E$ could not intersect more than one $(-2)$-curve of the resolution since two $(-1)$-curves can not intersects on a surface of general type. By Theorem \ref{thm:Serrano}(i) and (iii), then either $C_1$ or $C_2$ would be a Galois cover of $E\cong \mathbb P^1$ branched over two points; a contradiction since these  curves have both a genus greater or equal than $2$. 
\end{proof}
\section{The minimality of the remain list of product-quotient surfaces with $p_g=3$, $q=0$}\label{sec: minimality_PQ_pg3_q0}
In this section, we present an application of Theorem \ref{thm: min_Function_localization_-1_curves}. More precisely, we discuss the product-quotient surfaces with \(p_g=3\) and \(q=0\) listed in \cite[Table 21]{fedeLincei}, whose minimality had not yet been established. As a consequence of this work, we can improve \cite[Thm. 0.3]{fedeLincei} as follows: 
	\begin{theorem}\label{thm: improvement_classif_pg3_q0}
		Let $S$ be a regular product-quotient surface with $23\leq K_S^2\leq 32$ and $\chi(\mathcal O_S)=4$. Then $S$ is a surface of general type and it realizes one of the families of surfaces described in \cite[Tables 9--21]{fedeLincei}. Furthermore, these surfaces, with the exception of those in the families $no.  \ 537$, are minimal. 
	\end{theorem}
	\begin{proof}
		The proof follows from \cite[Thm. 0.3]{fedeLincei} and Proposition \ref{prop: MainThm_Min_pg3_q0} below. 
	\end{proof} 
The subsequent Table \ref{tab: minimality_results} consists of three columns: the first provides the identification number of each surface in the list \cite[Table 21]{fedeLincei}; the second contains the pair \((g(Y), Y^2)\) for each central component \(Y\) of \(S\), and the third reports the minimum of the function \(F\) in $\Delta_1\times \Delta_2$. 
\begin{table}[H]
	\centering
	 \begin{minipage}[t]{0.48\textwidth}
	 	\centering
	\begin{tabular}{|c|c|c|}
		\hline
		$no.$ & $(g(Y), Y^2)$ & $\min_{\Delta_1\times \Delta_2} F$ \\ \hline\hline
		523 & $(1,-1), (10,-1)$& 0\\ \hline
		524 & $(1,-1)$, $(10,-1)$& 0\\ \hline
		525 & $(46,-1)$, $(2,-1)$& $7/8$\\ \hline
		526 & $(3,-1)^2$, $(5,-1)^2$& $0$ \\ \hline
		527 & $(14,-2)$, $(0,-2)$ &$0$ \\ \hline
		528 & $(14,-2)$, $(0,-2)$ &$0$ \\ \hline
		529 & \makecell{$(28,-1)$, $(9,-1)$, \\ $(2,-2)$} & $9/20$ \\ \hline
		530 & $(14,-2)$, $(2,-2)$&$ 3/4$ \\ \hline
		531 & $(14,-2)$, $(2,-2)$&$ 3/4$ \\ \hline
		532 &\makecell{$(4,-1)$, $(3,-1)$, \\$(1,-2)$ }&0 \\ \hline
		533 & $(0,-2)$, $(6,-2)$& 0\\ \hline
		534 & $(1,-1)^2$, $(13,-2)$& 0\\ \hline
		535 & \makecell{$(14,-1)$, $(7,-1)$, \\ $(2,-1)$, $(1,-1)$}& 0\\ \hline
		536 &$(20,-2)$, $(0,-2)$ &0 \\ \hline
		537 & $(25,-1)$, $(0,-1)$& 0\\ \hline
		538 &$(13,-2)$,$(0,-2)$ & 0\\ \hline
		539 & $(3,-2)$, $(1,-2)$& 0\\ \hline
	\end{tabular}
\end{minipage}
\hfill 
\begin{minipage}[t]{0.48\textwidth}
	\centering
	\centering
	\begin{tabular}{|c|c|c|}
		\hline
		$no.$ & $(g(Y), Y^2)$ & $\min_{\Delta_1\times \Delta_2} F$ \\ \hline\hline
		540 & $(3,-2)$, $(1,-2)$& 0\\ \hline
		541 &$(13,-2)$,$(1,-2)$ & 0\\ \hline
		542 &$(5,-2)$, $(3,-2)$ & 3/8\\ \hline
		543 & \makecell{$(18,-1)$, $(6,-1)$, \\ $(5,-2)$}& 45/56 \\ \hline
		544 &$(13,-2)$, $(5,-2)$ & 9/8\\ \hline
		545 & $(19,-2)$, $(0,-2)$& 0\\ \hline
		546 & $(15,-1)$, $(2,-1)$& 0\\ \hline
		547 & $(9,-2)$, $(1,-1)^2$&0 \\ \hline
		548 & $(5,-1)^2$, $(1,-2)$ &0\\ \hline
		549 & $(13,-1)^2$, $(0,-2)$&0 \\ \hline
		550 & $(9,-2)$, $(1,-2)$& 0\\ \hline
		551 &$(25,-2)$, $(1,-1)^2$ &0  \\ \hline
		552 & $(9,-2)$, $(1,-1)^2$& 0\\ \hline
		553 & \makecell{$(62,-2)$, $(25,-1)$,\\  $(2,-2)$, $(1,-1)$}&0  \\ \hline
		554 & \makecell{$(62,-2)$, $(25,-1)$, \\ $(2,-2)$, $(1,-1)$}& 0\\ \hline
		555 & \makecell{$(62,-2)$, $(25,-1)$, \\ $(4,-2)$, $(1,-1)$}&0 \\ \hline
	\end{tabular}
\end{minipage}
\caption{Genus and self-intersection of the central components, and the minimum of the function \(F\) for each family of product-quotient surfaces in \cite[Table 21]{fedeLincei}.}
\label{tab: minimality_results}
\end{table}
\vspace{-0.5cm}

\begin{remark}
	All information in Table \ref{tab: minimality_results}, with the exception of cases $no.538$ and $no.553$, has been obtained by extracting the combinatorial structure (see the Introduction and Remark \ref{rem: F_only_dep_from_num_data}) from the pair of spherical systems of generators defining the corresponding family. More precisely, the data for each case in Table \ref{tab: minimality_results} and the \texttt{MAGMA} script used to compute their function \(F\), its minimum, and the genus and self-intersection of the central components are available at the following webpage:
	\begin{center}
		\href{https://github.com/Fefe9696/MAGMA_script_Minimality_Criterion_for_PQ_surfaces_of_general_type}{\texttt{MAGMA} script and results repository}
	\end{center}
However, due to computational restrictions, neither the number of families 
nor the defining spherical systems of generators are currently known for 
cases $no.538$ and $no.553$ in \cite[Table 21]{fedeLincei}. 
\begin{table}[H]
	\centering
	\hspace*{-1.8cm}
	\renewcommand{\arraystretch}{1.2}
	\begin{tabular}{| c| c| c|c| c| c|  c| }
		\hline
		$no.$ & $G(d,n)$ & Sing($X$) &	 $t_1$ & $t_2$  & $\Theta_1$ & $\Theta_2$ \\
		\hline
		\hline
		$ 538$ &$\langle 48,48\rangle$& $ 1/6, 1/2^2, 5/6$& $ 2, 4, 6$& $ 2^9, 6 $& $\frac{1}{12}$ &$\frac{10}{3}$\\
		$ 553$ &$\langle 160,234\rangle$& $ 1/5, 1/2^4, 4/5$& $ 2, 4, 5$& $ 2^4, 4, 5 $& $\frac{1}{20}$ &$\frac{31}{20}$\\
		\hline
	\end{tabular}
\end{table}

In any case, with some further analysis, one can deduce the combinatorial structure of the surface, namely that the entire basket of 
singularities must be contained in the points of ramification index $6$ 
for $no.538$, while for $no.553$ we have that $\frac{1}{5}$ and $\frac{4}{5}$ are contained in the points of ramification index $5$, and the remaining four nodes are contained in two points of ramification index $2$. 
\\
While this is already sufficient to compute the function $F$ as well as 
the genus and self-intersection of the central components, Corollary 
\ref{cor: improvement_Lemma_Roberto} allows us to automatically deduce 
that they yield minimal surfaces of general type.
\end{remark}
\begin{proposition}\label{prop: MainThm_Min_pg3_q0}
	Let $S$ be a regular product-quotient surface of general type in the list \cite[Table 21]{fedeLincei} with the exception of  $no.  \ 537$. Then $S$ is minimal. 
\end{proposition}
\begin{proof}
The result follows directly from Theorem \ref{thm: min_Function_localization_-1_curves}, as shown in Table \ref{tab: minimality_results}, with the only exception of case $no.544$. In this instance, the minimum of the function \(F\) is \(\frac{9}{8}\), which is greater than one; thus, Theorem \ref{thm: min_Function_localization_-1_curves} does not hold. Let us consider this case in further detail:
	\begin{table}[H]
		\centering
		\hspace*{-1.8cm}
		\renewcommand{\arraystretch}{1.2}
		\begin{tabular}{| c| c| c|c| c| c|  c| }
			\hline
			$no.$ & $G(d,n)$ & Sing($X$) &	 $t_1$ & $t_2$  & $\Theta_1$ & $\Theta_2$ \\
			\hline
			\hline
			$ 544$ &$\langle 768, 1086051\rangle$& $ 1/6, 1/2^2, 5/6$& $ 2, 4, 6$& $ 2, 6, 8 $& $\frac{1}{12}$ &$\frac{5}{24}$\\
			\hline
		\end{tabular}
	\end{table}

One can check that there are only two central components, \(Y_{1}\) and \(Y_{2}\), over the points of \(C_1/G\) and \(C_2/G\) with ramification index 6. We index the exceptional divisors \(E_{i}\) of the resolution of \(X\) according to the order of the singularities in the basket (for instance, \(E_{1}\) is the \((-6)\)-rational curve, and so on). From Proposition \ref{prop: K_S_as_A-B}, we have \(K_S \equiv A-B\), where:
\[
\begin{split}
	A & :=\frac{1}{2}Y_1+\frac{5}{4}Y_2+\frac{7}{8}E_2+\frac{7}{8}E_3+\frac{5}{8}E_4+\frac{3}{4}E_5
	+\frac{7}{8}E_6+E_7+\frac{9}{8}E_8\\ 
	B& :=\frac{3}{8}E_1.
\end{split}\]
Let us assume there exists a \((-1)\)-curve \(E\) on \(S\). Applying Proposition \ref{prop: diseg_intersezione} to \(E_{1}\), we obtain \(E_1 E \leq \max\{1, 6-3\}=3\). Let us suppose that \(E_1  E \leq 2\); then we would have:
\[
A E=(K_S+B) E= -1+\frac{3}{8}E_1E\leq -1+ \frac{3}{4}=-\frac{1}{4}<0.
\]
This implies that \(E\) would be in the support of \(A\), a contradiction since both central components have self-intersection \(-2\).
This forces \(E_1E = 3\), so $AE=\frac{1}{8}$. 
 Since \(E\) is not a component of \(A\), its intersection with each component of the support of \(A\) must be non-negative. Furthermore, \(E\) must intersect at least one of these components positively, otherwise \(AE\) would be zero. This leads to:
\[
\frac{1}{8}=A E\geq \min\left\{\frac{1}{2}, \frac{5}{4},\frac{7}{8}, \frac{5}{8}, \frac{3}{4},1, \frac{9}{8}\right\}=\frac{1}{2},
\]
a contradiction. Thus, \(E\) can not exist and \(S\) is a minimal surface.
\end{proof}
\begin{theorem}\label{thm: MainThm_note_2}
	Surfaces belonging to one of the families  $no.537$ are not minimal. They have a $(-1)$-curve contained in a fibre of $f_2\colon S\to C_2/G$, which intersects transversally in one point the exceptional $(-2)$-curve of the resolution of the node of the quotient model $X=\left(C_1\times C_2\right)/G$. After contracting both of these curves we obtain a minimal surface of general type with $K^2=26$. 
\end{theorem}
\begin{proof}
	Surfaces belonging to the families $no.537$ share the following information: 
		\begin{table*}[h]
		\centering
		\hspace*{-1.8cm}
		\renewcommand{\arraystretch}{1.2}
		\begin{tabular}{| c| c| c|c| c| c|  c| }
			\hline
			$no.$ & $G(d,n)$ & Sing($X$) &	 $t_1$ & $t_2$  & $\Theta_1$ & $\Theta_2$ \\
			\hline
			\hline
			$ 537$ &$\langle 192, 181\rangle$& $ 1/8^2, 1/4, 1/2$& $ 2, 3, 8$& $ 2^2, 4^3, 8 $& $\frac{1}{24}$ &$\frac{17}{8}$\\
			\hline
		\end{tabular}
	\end{table*}

	One can easily verify that the entire basket of singularities is contained in the points of ramification index $8$. Moreover, Proposition \ref{prop: loc_-1_curves_from_signatures} holds and so only the two central components $Y_1$ and $Y_2$ may be $(-1)$-curves of $S$. Indeed, we notice from Table  \ref{tab: minimality_results} that $Y_2$ is a $(-1)$-curve. 
From Theorem \ref{thm:Serrano}$(ii)$, $Y_2$ is intersecting transversally in one point  the $(-2)$-curve $E_4$ of the resolution. We contract the curve $Y_2$, so the strict transform of the curve $E_4$ becomes a $(-1)$-curve. We contract also the curve $E_4$ and we obtain that the numerical canonical class of the blow-down $b\colon S\to S'$ is then 
	\begin{equation}\label{eq: can_div_of_contact_Y2_and_E4}
		K_{S'}\equiv\frac{1}{3}b(Y_1)+\frac{17}{12}(b(E_1)+b(E_2))+\frac{23}{6}b(E_3).
	\end{equation}
	In this case, we have that 
	\[
	\begin{split}
		b^*(b(Y_1))=Y_1+5Y_2+3E_4, &\qquad b^*(b(E_1))=E_1+2Y_2+E_4\\
		b^*(b(E_2))=E_2+2Y_2+E_4, &\qquad b^*(b(E_3))=E_3+2Y_2+E_4,
	\end{split}
	\]
	so the intersection pair matrix among $b(Y_1), b(E_1), b(E_2)$ and $b(E_3)$ is equal to 
	\begin{table}[H]
		\centering
		\begin{tabular}{|c|c|c|c|c|}
			\hline
			$\cdot$ & $b(Y_1)$ & $b(E_1)$ & $b(E_2)$ & $b(E_3)$ \\ \hline
			$b(Y_1)$  & $12$  & $6$ & $6$ & $6$ \\ \hline
			$b(E_1)$  & $6$ & $ -6$ & $2$ & $2$ \\ \hline
			$b(E_2)$ & $6$  & $ 2$ & $-6$ & $2$ \\ \hline
			$b(E_3)$ & $6$ & $2$  & $2$ & $ -2$  \\ \hline
		\end{tabular}
		\label{tab: int_prod}
	\end{table}
Thus, $K_{S'}$ is a nef divisor, as it is effective and its intersection with any of its components \eqref{eq: can_div_of_contact_Y2_and_E4} is nonnegative.
\end{proof}
\section{The minimality of product-quotient surfaces of general type with $p_g=q=0$ and $K^2\geq -1$}\label{sec: minimality_PQ_pg0_q0}
%
In this short section, we briefly discuss the minimality of product-quotient surfaces with $p_g=q=0$ and $K^2\geq -1$. These surfaces were classified in \cite[Tables 1 and 2, Sec. 5]{BP12} and \cite[Subsec. 7.1 and 7.2]{BP16}, where their minimality was already addressed using a different approach.  

In Table \ref{tab: cases_pg=q=0_Ksquare_geq_-1}, we list the self-intersection of the central components of $S$ for each case in the classification where the minimum of the function $F$ is non-zero, and the case with $G\cong \mathbb Z_5^2$. 
\\
As a result, Theorem \ref{thm: min_Function_localization_-1_curves} confirms the minimality except for the last three cases with $G\cong \PSL(2,7)$, $G\cong \GL(2,\mathbb{Z}_4)$, and $G\cong \mathbb Z_5^2$. 
\\
For the first two exceptions, Theorem \ref{thm: min_Function_localization_-1_curves} does not apply. Indeed, as proved in \cite[Sec. 5]{BP12} and \cite[Subsec. 7.1]{BP16}, the surface $S$ admits $(-1)$-curves that are not contained in the two natural fibrations of $S$. On the other hand, in the last case where $G\cong \mathbb Z_5^2$ Theorem \ref{thm: min_Function_localization_-1_curves} holds and indeed both central components of $S$ are $(-1)$-curves. This latter case is studied in detail in \cite[Subsec. 7.2]{BP16}. 
  \begin{table}[h]
  		\centering
  		\begin{tabular}{|c|c| c|c|}
  			\hline
  			$K^2_S$ & $G$ & $(g(Y), Y^2)$ & $\min_{\Delta_1\times \Delta_2} F$ \\ \hline\hline
  			4 & $\mathfrak A_6$& $(5,-1), (2,-1)$& 1/4\\ \hline
  			3 & $\mathfrak A_5$& $(2,-1)$, $(1,-1)$& 1/3\\ \hline
  			3 & $\mathbb Z_2^4\rtimes D_5$& $(5,-1)$, $(1,-1)$& $1/2$\\ \hline
  			3& $\mathfrak A_6$ & $(5,-1)$, $(2,-1)$& $5/6$ \\ \hline
  			2 & $\mathbb Z_5^2\rtimes \mathbb Z_3$ & $(2,-1)^4$&$ 2/5$ \\ \hline
  			2 & $\mathfrak A_5$ & $(2,-1)^2$, $(1,-2)$&$ 1/5$ \\ \hline
  			2& $\PSL(2,7)$ & \makecell{$(4,-1)$ , $(2,-1)$, \\ $(2,-2)$}&$13/42$ \\ \hline
  			2& $\mathfrak A_6$& \makecell{$(8,-1)$, $(4,-1)$, \\$(2,-2)$} &$11/30$ \\ \hline
  			2 & $\mathfrak S_5$ & \makecell{$(8,-1)$, $(4,-1)$, \\ $(1,-1)^2$} & $1/5$ \\ \hline
  			1&$\PSL(2,7)$ & \makecell{$(7,-2)$, $(5,-1)$, \\$(1,-1)^3$ }&1/7 \\ \hline
  			1 & $\PSL(2,7)$ & $(2,-1)$, $(1,-1)$& 23/12\\ \hline
  			0& $\GL(2,\mathbb Z_4)$ & $(2,-2)$, $(0,-2)$& 3/2\\ \hline
  			-1 & $\mathbb Z_5^2$ & $(0,-1)^2$& 0\\ \hline
  		\end{tabular}
  	 \caption{}
  \label{tab: cases_pg=q=0_Ksquare_geq_-1}
\end{table}
\section*{Acknowledgements}

The author wish to thank F. Catanese and R. Pignatelli for fruitful discussions.  The author held a research fellowship
from INdAM, Istituto Nazionale di Alta Matematica and was partially supported by the Research Project \textit{Surfaces of general type and their geometry}, CUP E53C25002010001. The author is a member of the INDAM-GNSAGA.%

\end{document}